\documentclass[12pt]{amsart}
\usepackage{latexsym,amssymb,amscd}

\def\B'c{{\mathcal{B'}}}
\def\U'c{{\mathcal{U'}}}

\def\opn#1#2{\def#1{\operatorname{#2}}} % to make operators
\opn\chara{char}
\opn\length{\ell}
\opn\cd{cd}
\opn\projdim{proj\,dim}
\opn\injdim{inj\,dim}
\opn\ini{in}
\opn\rank{rank}
\opn\depth{depth}
\opn\height{ht}
\opn\bigheight{bight}
\opn\embdim{emb\,dim}
\opn\codim{codim}

\opn\Tr{Tr}
\opn\bigrank{big\,rank}
\opn\superheight{superheight}\opn\lcm{lcm}
\opn\trdeg{tr\,deg}%
\opn\reg{reg}
\opn\lreg{lreg}
\opn\set{set}
\opn\supp{Supp}
\opn\shad{Shad}
\opn\indeg{indeg}
\opn\lex{lex}
\opn\div{div}
\opn\Div{Div}
\opn\cl{cl}
\opn\Cl{Cl}
\opn\Spec{Spec}
\opn\Supp{Supp}
\opn\supp{supp}
\opn\Sing{Sing}
\opn\Ass{Ass}
\opn\Ann{Ann}
\opn\Rad{Rad}
\opn\Soc{Soc}
\opn\Ker{Ker}
\opn\Coker{Coker}
\opn\Im{Im}
\opn\Hom{Hom}
\opn\Tor{Tor}
\opn\Ext{Ext}
\opn\End{End}
\opn\Aut{Aut}
\opn\id{id}

\opn\nat{nat}
\opn\GL{GL}
\opn\SL{SL}
\opn\mod{mod}
\opn\ord{ord}
\opn\ara{ara}
\opn\aff{aff}
\opn\con{conv}
\opn\relint{relint}
\opn\st{st}
\opn\lk{lk}
\opn\cn{cn}
\opn\core{core}
\opn\vol{vol}
\opn\Ind{Ind}
\opn\gr{gr}

\def\pot#1#2{#1[\kern-0.28ex[#2]\kern-0.28ex]}

\opn\dirlim{\underrightarrow{\lim}}
\opn\invlim{\underleftarrow{\lim}}
\def\pnt{{\raise0.5mm\hbox{\large\bf.}}}

\def\Implies{\ifmmode\Longrightarrow \else
     \unskip${}\Longrightarrow{}$\ignorespaces\fi}
\def\implies{\ifmmode\Rightarrow \else
     \unskip${}\Rightarrow{}$\ignorespaces\fi}
\def\iff{\ifmmode\Longleftrightarrow \else
     \unskip${}\Longleftrightarrow{}$\ignorespaces\fi}

\let\:=\colon
\newtheorem{Theorem}{Theorem}[section]
\newtheorem{Lemma}[Theorem]{Lemma}
\newtheorem{Corollary}[Theorem]{Corollary}
\newtheorem{Proposition}[Theorem]{Proposition}

\let\epsilon=\varepsilon
\let\phi=\varphi
\let\kappa=\varkappa
\numberwithin{equation}{section}

\title{The monomial ideal of independent sets associated to a graph}
\author{Oana Olteanu}

\address{University ``Politehnica" of Bucharest, Faculty of Applied Sciences,
Splaiul Independen\c tei, No.
313, 060042, Bucharest, Romania}
\email{olteanuoanastefania@gmail.com} 

\begin{document}

\maketitle

\begin{abstract} Independent sets play a key role into the study of graphs and important problems arising in graph theory reduce to them. We define the monomial ideal of independent sets associated to a finite simple graph and describe its homological and algebraic invariants in terms of the combinatorics of the graph.
We compute the minimal primary decomposition and characterize the Cohen--Macaulay ideals. Moreover, we provide a formula for computing the Betti numbers, which depends only on the coefficients of the independence polynomial of the graph. \\



\end{abstract}

\section*{Introduction}
Graph theory has various applications in many fields, such as computer sciences, statistics and also biology or chemistry. Let $G$ be a simple graph on the vertex set $V(G)$ and the set of edges $E(G)$. An independent set of $G$ is a set of vertices such that there are no two vertices adjacent in $G$. In the literature, independent sets are also called stable sets, and many important problems arising in graph theory can be stated in terms of them. 

We consider two polynomial rings $R=K[r_S: S\mbox{ independent set of }G] \mbox{ and }T=K[s_i,t_i:i\in V(G)]$ over a field $K$, and the ring homomorphism $\phi:R\rightarrow T$ given by $\phi(r_S)=\prod\limits_{i\in S}s_i\prod\limits_{i\notin S}t_i$. It is customary to define the toric ideal $I_G=\ker(\phi)$, which generalizes several classical examples of toric ideals. These ideals have been intensively studied and have important applications in algebraic statistics.

In this paper we define the \textit{monomial ideal of independent sets} associated to the graph $G$ to be the squarefree monomial ideal generated by the monomials $\phi(r_S)$, where $S$ are independent sets of $G$. In fact, the name of these ideals was suggested by \cite{EN}. In their paper \cite{EN}, Engstr\"om and Nor\'en described the independent sets of $G$ as graph homomorphisms from $G$ to the graph with two vertices, one edge and a loop (also called the independence target graph). The ideal of graph homomorphisms is the ideal of independent sets.

For monomial ideals of independent sets, we aim at studying the homological and algebraic invariants of them.  

The paper is structured as follows. The first section represents an overview of the notions and concepts that will be used in this paper. We briefly recall the definition of some numerical invariants attached to a monomial ideal, expressed in terms of the minimal graded free resolution of the ideal, and some useful results related to them.

Section $2$ is the main section of this paper and we describe the properties of the monomial ideals of independent sets. Throughout Section $2$, we characterize the algebraic and homological invariants of the monomial ideal of independent sets associated to graphs. We begin by determining the minimal primary decomposition in Theorem \ref{prim dec}. The minimal primes correspond to the sets of vertices and edges, being of the form $(s_i,t_i)$, with $i\in V(G)$ and $(t_i,t_j)$, with $\{i,j\}\in E(G)$. Next, we prove in Theorem \ref{linear quotients} that the monomial ideal of independent sets associated to a finite simple graph has linear quotients with respect to a given order of its minimal monomial generators, hence it has a linear resolution. As a consequence, in Corollary \ref{invariants}, for the monomial ideal of independent sets of a graph we determine the Krull and projective dimensions, the Betti numbers and the Castelnuovo--Mumford regularity. Moreover, we characterize the monomial ideals of independent sets which are Cohen--Macaulay, by using Alexander duality, in Theorem \ref{dual}. We pay a special attention to the combinatorial information stored by the graph, thus all the invariants are expressed in terms of the combinatorics of the graph. 

In the last section, Section $3$, we analyze some particular classes of graphs. As it follows from previous section, the Betti numbers and the projective dimension of monomial ideal of independent sets are characterized by the number of independent sets of a given cardinality. This information is stored by the independence polynomial, more precisely, by the coefficients of the independence polynomial associated to a graph. Unfortunately, these coefficients are not known for an arbitrary graph. Hence we will apply in our study the results obtained in \cite{A}, \cite{BH}, \cite{LM} and \cite{LM1} for paths, cycles, powers of cycles and centipede graphs. These will allow us to explicitly compute the Betti numbers and the projective dimension of the monomial ideal of independent sets for the mentioned particular classes of graphs.

\section{Preliminaries}
Let $G$ be a simple graph on the vertex set $V(G)=[n]=\{1,2,\ldots,n\}$ and the set of edges $E(G)$. We recall that a set $S$ of vertices of $G$ is \textit{independent} if there are no two elements $i$ and $j$ of $S$ such that $\{i,j\}\in E(G)$. We denote by $\Ind(G)$ the set of all the independent sets of $G$. Let $\alpha(G)$ be the maximal cardinality of an independent set of $G$, called the \textit{independence number} of the graph $G$. 

Important aspects of the combinatorial information of a graph are stored by the independence polynomial, defined by Gutman and Harary \cite{GH}. The \textit{independence polynomial} of the graph $G$ is $I(G;x)=\sum\limits_{i=0}^{\alpha(G)}s_jx^j$, where $s_j$ is the number of independent sets of cardinality $j$ in the graph $G$ and $s_0=1$. The independence polynomial was defined as a natural generalization of the matching polynomial of a graph. Independence polynomials play a key role in studying statistical physics and combinatorial chemistry. This notion will be useful in Section $3$. We refer the reader to \cite{GH}, \cite{LM}, \cite{LM1} for more information concerning the independence polynomial.

In the sequel, we consider the polynomial rings over a field $K$:
$$R=K[r_S: S\in \Ind(G)] \mbox{ and }T=K[s_i,t_i:i\in V(G)],$$
and the ring homomorphism $\phi:R\rightarrow T$ given by $\phi(r_S)=\prod\limits_{i\in S}s_i\prod\limits_{i\notin S}t_i$. Let $I\subset T$ be the squarefree monomial ideal generated by the monomials $\phi(r_S)$, where $S$ is an independent set of $G$. We call the ideal $I$ the \textit{monomial ideal of independent sets} associated to the graph $G$. We will characterize the homological and algebraic invariants of the ideal $I$, using the combinatorial data enclosed by the graph $G$.

Furthermore, we recall the main invariants of the monomial ideals, which we will study later in this paper. Let $S=K[x_1,\ldots,x_n]$ be the polynomial ring in $n$ variables over the field $K$. We order the monomials in $S$ lexicographically with $x_1>_{lex}\cdots>_{lex}x_n$. For a monomial ideal $I\subset S$, we will denote by $G(I)$ the set of minimal monomial generators of $I$. The minimal graded free resolution of $I$ over $S$ is:
$$0\rightarrow\bigoplus_{j}S(-j)^{\beta_{p,j}}\rightarrow\ldots \rightarrow\bigoplus_{j}S(-j)^{\beta_{1,j}}\rightarrow\bigoplus_{j}S(-j)^{\beta_{0,j}}\rightarrow I\rightarrow 0.$$
The \textit{Betti numbers} of $I$ are defined by $\beta_i(I)=\sum_{j}\beta_{i,j}(I)$. The \textit{projective dimension} of $S/I$ is $\projdim(S/I)=\max\{i:\beta_{i,j}(S/I)\neq 0, \mbox{ for some }j\}=p+1$, and the \textit{Castelnuovo--Mumford regularity} of $I$ is given by $\reg(I)=\max\{j-i:\beta_{i,j}(I)\neq 0\}$.

We recall that a monomial ideal $I\subset S$ has \textit{linear quotients} if there is an order of the minimal monomial set of generators of $I$, $u_1,\ldots, u_s$ such that for all $2\leq i\leq s$ the colon ideals $(u_1,\ldots,u_{i-1}):u_i$ are generated by variables. In this case, we will denote by $\set(u_i)=\{x_j:x_j\in(u_1,\ldots,u_{i-1}):u_i\}$.

A useful formula to compute the Betti numbers of ideals with linear quotients was given in \cite{HH}:

\begin{Proposition}\cite{HH}\label{betti}
Let $I\subset S$ be a graded ideal with linear quotients, generated in one degree. Then
	\[\beta_i(I)=\sum\limits_{u\in G(I)}{|\set(u)|\choose i}.
\]
In particular, it follows that $\projdim(I)=\max\{|\set(u)|:u\in G(I)\}$.
\end{Proposition} 

It is known that any monomial ideal generated in one degree, which has linear quotients, has a linear resolution, \cite{CH}. In \cite{HHZ}, the monomial ideals generated in degree $2$ with a linear resolution are described.

\begin{Theorem}\cite{HHZ}\label{powers}
Let $I$ be a monomial ideal generated in degree $2$. The following conditions are equivalent:
\begin{itemize}
\item [(a)] $I$ has a linear resolution;
\item [(b)] $I$ has linear quotients;
\item [(c)] Each power of $I$ has a linear resolution.
\end{itemize}
\end{Theorem}

A very useful tool in characterizing the squarefree monomial ideals which are Cohen--Macaulay is the Eagon--Reiner theorem \cite{ER}. The result uses concepts such as simplicial complex, Stanley--Reisner ideal and Alexander duality, for which we refer the reader to \cite{HH}.

\begin{Theorem}[Eagon--Reiner, \cite{ER}]\label{Eagon--Reiner}
Let $\Delta$ be a simplicial complex on $[n]$. Then the Stanley--Reisner ideal $I_{\Delta}\subset S$ has a linear resolution if and only if $S/I_{\Delta^\vee}$ is Cohen--Macaulay. More precisely, $I_{\Delta}$ has a $q-$linear resolution if and only if $S/I_{\Delta^\vee}$ is Cohen--Macaulay of dimension $n-q$.
\end{Theorem}
Here $\Delta^\vee$ means the Alexander dual of $\Delta$. In this paper, we will also denote $I_{\Delta^\vee}$ by $I^\vee$.

\section{Invariants of the monomial ideal of independent sets} 
This section is devoted to determining some algebraic and homological invariants of the monomial ideal of independent sets associated to a graph $G$. The characterizations aim to point out the combinatorial aspects of the graph. We start our study with the standard decomposition of monomial ideals of independent sets. 

\begin{Theorem}\label{prim dec}
Let $G=(V(G),E(G))$ be a finite simple graph and $I\subset T$ be the monomial ideal of independent sets, with its minimal monomial generating set $$G(I)=\{m_S=\prod\limits_{i\in S}s_i\prod\limits_{i\notin S}t_i| S\in \Ind(G)\}.$$
Then the minimal primary decomposition of $I$ is 
$$I=\bigcap\limits_{i\in V(G)}(s_i,t_i)\cap \bigcap\limits_{\{i,j\}\in E(G)}(t_i,t_j).$$
\end{Theorem}

\begin{proof}
Let $\frak p$ be a minimal prime ideal which contains $I$. There is an integer $i\in V(G)$ such that $t_i\in \frak p$, since $\emptyset$ is an independent set of the graph $G$. The set $\{i\}$ being independent implies that $s_i\in \frak p$ or $s_i\notin \frak p$ and $t_j\in \frak p$, for some $j\in V(G)$. We analyze these two cases:

\textit{Case 1:} If $s_i \in \frak p$, then we conclude that $\frak p\supseteq(s_i,t_i)\supset I$, and the equality follows by the minimality of $\frak p$.

\textit{Case 2:} Assume that $s_i\notin \frak p$ and $t_j\in \frak p$, for some $j\in V(G)$. 

We claim that $(t_i,t_j)$ is a minimal prime ideal of $I$ if and only if $\{i,j\}\in E(G)$. Indeed, assume by contradiction that $\{i,j\}\notin E(G)$. Then there is an independent set $S$ of the graph such that $\{i,j\}\in S$. These implies that the monomial $m_S \in I$ and $m_S\notin (t_i,t_j)$, contradiction. Conversely, if $\{i,j\}\in E(G)$, then $i$ and $j$ cannot be both in the same independent set. Hence any monomial $m_S \in G(I)$ is divisible at least by one of $t_i$ or $t_j$, thus $(t_i,t_j)\supset I$ is a minimal prime ideal.

Therefore $\frak p\supseteq (t_i,t_j)\supset I$, which ends the proof. 
\end{proof}

One may note that the minimal primary decomposition of a monomial ideal of independent sets associated to a graph $G$ can be written just looking to the sets of edges and vertices of the graph, as it follows from Theorem \ref{prim dec}.

Furthermore, we will use the following notations. Let $G=(V(G),E(G))$ be a finite simple graph and $I\subset T=K[s_i,t_i: i\in V(G)]$ be the squarefree monomial ideal with $$G(I)=\{m_i=\prod\limits_{r\in S_i}s_r\prod\limits_{r\notin S_i}t_r| S_i\in \Ind(G)\}$$ 
its minimal monomial generating set. For a monomial $m_i=\prod\limits_{r\in S_i}s_r\prod\limits_{r\notin S_i}t_r\in G(I)$, we denote by $m_i^{(s)}$ and $m_i^{(t)}$ the $s-$part and $t-$part, namely $m_i^{(s)}=\prod\limits_{r\in S_i}s_r$, $m_i^{(t)}=\prod\limits_{r\notin S_i}t_r$ respectively. Moreover, by $\deg_s(m_i)$ and $\deg_t(m_i)$ we refer to the degree of the monomials $m_i^{(s)}$ and $m_i^{(t)}$. We consider the lexicographical order on the monomials in $K[s_i: i\in V(G)=\{1,2,\ldots,n\}]$ with $s_1>_{\lex}s_2>_{\lex}\ldots >_{\lex} s_n$. By $m_i>_{lex}m_j$ we mean that $m_i^{(s)}>_{\lex}m_j^{(s)}$. Next, we define the following monomial order on the monomials in $T$: $m_i\succ m_j \mbox{ if and only if }\deg_s(m_i)<\deg_s(m_j) \mbox { or, } \deg_s(m_i)=\deg_s(m_j) \mbox{ and } m_i>_{lex}m_j.$

With these notations, one has:

\begin{Theorem}\label{linear quotients}
Let $I\subset T$ be the monomial ideal of independent sets whose minimal monomial generating set $G(I)=\{m_i=\prod\limits_{r\in S_i}s_r\prod\limits_{r\notin S_i}t_r: S_i \in \Ind(G)\}$ is ordered decreasing in the $\prec$ order. Then $(m_1,\ldots,m_{i-1}):(m_i)=(t_r: r\in S_i)$, for all $i>1$.  
\end{Theorem}

\begin{proof}
Let $i>1$ and $M$ be a monomial in the colon ideal $(m_1,\ldots,m_{i-1}):(m_i)$. Then there is some $j<i$ such that $m_j\mid Mm_i$, where $m_j\in G(I)$. Since $m_j\succ m_i$, we have to distinguish between the following two cases:

\textit{ Case 1:} If $\deg_s(m_j)<\deg_s(m_i)$, then there is an integer $r\in S_i$ such that $r\notin S_j$. We obtain that $t_r\mid m_j$ and $t_r\nmid m_i$. Therefore $t_r\mid M$ and $M\in (t_r: r\in S_i)$.

\textit{Case 2:} If $\deg_s(m_i)=\deg_s(m_j)$ and $m_j>_{lex}m_i$, then we denote $m_j^{(s)}=s_{r_1}\cdots s_{r_p}$ and $m_i^{(s)}=s_{q_1}\cdots s_{q_p}$. Since $m_j>_{\lex} m_i$, it follows that $r_1=q_1,\ldots, r_c=q_c$ and $r_{c+1}<q_{c+1}$ for some $c>0$. By degree consideration, $s_{r_l}\nmid m_j^{(s)}$, for some $l\geq c+1$. It results that $t_{r_l}\mid m_j$ and $t_{r_l}\nmid m_i$, thus $M\in (t_r: r\in S_i)$.
 
Conversely, let $r\in S_i$ and consider the monomial $m_j=t_rm_i/s_r$. Then it is clear that $m_j\in G(I)$ since $S_i\setminus\{r\}$ it remains an independent set of $G$. Moreover, $\deg_s(m_j)<\deg_s(m_i)$ implies that $m_j\succ m_i$. Therefore $t_r \in (m_1,\ldots,m_{i-1}):(m_i)$, which ends the proof.
\end{proof}

As a consequence, one may note that if we order the minimal monomial generating set $G(I)=\{m_i: S_i \in \Ind(G)\}$ decreasing in the $\prec$ order, then $|\set(m_i)|$ equals the cardinality of the independent set $S_i$. We denote by $s_k$ the number of independent sets with $k$ elements of the graph $G$.

From the previous two results, we obtain the following characterization of the numerical invariants of the monomial ideal of independent sets associated to the graph $G$. 

\begin{Corollary}\label{invariants}
In the same hypothesis, one has:
\begin{itemize}
	\item [(a)] $I$ has a linear resolution;
	\item [(b)] The Castelnuovo--Mumford regularity of $I$ is $\reg(I)=|V(G)|$;
	\item [(c)] The Betti numbers of $I$ are $\beta_i(I)=\sum\limits_{k=0}^{\alpha(G)}s_k{k \choose i}$, for $i\geq 0$;
	\item [(d)] The projective dimension of $T/I$ is $\projdim(T/I)=\alpha(G)+1;$
	\item [(e)] The Krull dimension of $T/I$ is $\dim(T/I)=2|V(G)|-2$;
	\item [(f)] $T/I$ is Cohen--Macaulay if and only if $G$ is the complete graph. 
\end{itemize}
\end{Corollary}

\begin{proof} The statement $(a)$ follows by Theorem \ref{linear quotients}. In particular, we immediately obtain the formula for the Castelunovo--Mumford regularity.

In order to prove $(c)$, we apply Theorem \ref{betti} and use the notations mentioned above. 

For $(d)$, we have
$$\projdim(T/I)=\max\{i: \beta_i(T/I)\neq 0\}=\max\{i: \beta_{i-1}(I)\neq 0\}=\alpha(G)+1.$$
%$$=\max\{|S_j|: S_j \mbox{ is a maximum independent set of }G\}+1=\alpha(G)+1.$$

The statement $(e)$ is a consequence of the primary decomposition, Theorem \ref{prim dec}.

At $(f)$, one has that $T/I$ is Cohen--Macaulay if and only if $\alpha(G)+1=2$, equivalently $G$ is the complete graph.
\end{proof}

As it follows from the previous result, the invariants of the monomial ideal of independent sets can be computed if one knows all the cardinalities of the independent sets of the graph, respectively its independence number. Actually, the problem of computing the independence number of an arbitrary graph is fundamental in theoretical computer science. In fact, for the maximum independent set problem, some approximation algorithms were given only for particular classes of graphs. A more difficult question is to determine the number of all the independent sets in a graph, on the same cardinality.

Furthermore, the Cohen--Macaulay monomial ideals of independent sets can be characterized, using Alexander duality:

\begin{Theorem}\label{dual}Let $I\subset T$ be the monomial ideal of independent sets associated to a graph $G$.
The following statements are equivalent:
\begin{itemize}
	\item [(a)] $G$ is the complete graph;
	\item [(b)] $T/I$ is Cohen--Macaulay;
	\item [(c)] The Alexander dual of $I$, $I^\vee$, has a linear resolution;
	\item [(d)] $I^\vee$ has linear quotients;
	\item [(e)] All the powers of $I^\vee$ have a linear resolution.
\end{itemize}
\end{Theorem}

\begin{proof}
By Corollary \ref{invariants} (f) and Eagon--Reiner Theorem \cite{ER}, we obtain $(a)\Leftrightarrow(b)\Leftrightarrow(c)$. The equivalences $(c)\Leftrightarrow(d)\Leftrightarrow(e)$ follows by Theorem \ref{powers}.
\end{proof}

\section{Applications on some classes of graphs}

In this section, we will analyze some particular classes of graphs. As it follows from Corollary \ref{invariants}, the invariants of the monomial ideal of independent sets of a graph may be computed if the cardinalities of independent sets are known. Since all the information concerning the number of the independent sets of a given cardinality in a graph is enclosed in the independence polynomial, we have to consider some particular classes of graphs for which the independence polynomial was computed. We will use especially the results from \cite{A}, \cite{BH}, \cite{LM} and \cite{LM1}.

\subsection{The path graph $P_n$}

Let $G=P_n$ be the path graph on the vertex set $[n]$, with $n\geq 1$. In \cite{A}, J.L. Arocha derived a formula for the independence polynomial of this class of graphs, in terms of Fibonacci polynomials. In the following, we will use the description given in \cite[pp. 234]{LM}, for the coefficients of the independence polynomial of the graph $P_n$:
$$s_k={n+1-k\choose k}, \mbox{ for }0\leq k\leq \left\lfloor\frac{n+1}{2}\right\rfloor.$$

By applying this result we obtain:

\begin{Proposition}
Let $G=P_n$ be the path graph on $n$ vertices, $n\geq 1$. Let $I\subset T=K[s_i,t_i: 1\leq i\leq n]$ be the monomial ideal of independent sets, with the minimal monomial generating set $G(I)=\{m_i=\prod\limits_{r\in S_i}s_r\prod\limits_{r\notin S_i}t_r| S_i\in \Ind(P_n)\}$. Then the Betti numbers of $I$ are
$$\beta_i(I)=\sum\limits_{k=0}^{\left\lfloor\frac{n+1}{2}\right\rfloor}{n+1-k \choose k}{k \choose i}, \mbox{ for }i\geq 0,$$
and the projective dimension of the quotient ring $T/I$ is $\projdim(T/I)=\left\lfloor\frac{n+1}{2}\right\rfloor+1.$
\end{Proposition}

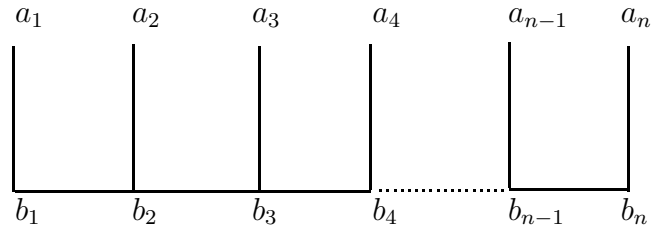
\subsection{The centipede graph}

The centipede graph $W_n$, with $n\geq 1$, is the graph on the vertex set $\{a_1,\ldots,a_n\}\cup\{b_1,\ldots,b_n\}$. The set of edges of the centipede graph is given by $E(W_n)=\{{a_i,b_i}:1\leq i\leq n\}\cup\{\{b_j,b_{j+1}:1\leq j\leq n-1\}$.

	\[
\]
\begin{center}
	
	%TeXCAD (http://texcad.sf.net/) Picture. File: [centipede graph.pic]. Options on following lines.
%\grade{\on}
%\emlines{\off}
%\epic{\off}
%\beziermacro{\on}
%\reduce{\on}
%\snapping{\off}
%\pvinsert{% Your \input, \def, etc. here}
%\quality{8.000}
%\graddiff{0.005}
%\snapasp{1}
%\zoom{4.0000}
\unitlength 1mm % = 2.845pt
\linethickness{0.8pt}
\ifx\plotpoint\undefined\newsavebox{\plotpoint}\fi % GNUPLOT compatibility
\begin{picture}(86.25,29.25)(0,0)
\put(4.5,6.75){\line(1,0){47.25}}
%\dottedline(53,6.75)(69.25,6.75)
\multiput(52.93,6.68)(.95588,0){18}{{\rule{.8pt}{.8pt}}}
%\end
\put(70.25,7){\line(1,0){15.75}}
\put(4.25,6.75){\line(0,1){19.25}}
\put(20.25,7){\line(0,1){19.25}}
\put(37,6.75){\line(0,1){19.25}}
\put(51.75,7){\line(0,1){19.25}}
\put(70.25,7.25){\line(0,1){19.25}}
\put(86,6.75){\line(0,1){19.25}}
\put(4.5,29.25){$a_1$}
\put(20,29.25){$a_2$}
\put(36,29.25){$a_3$}
\put(52,29.25){$a_4$}
\put(70,29.25){$a_{n-1}$}
\put(85,29.25){$a_n$}
\put(85,3){$b_n$}
\put(70,3){$b_{n-1}$}
\put(52,3){$b_4$}
\put(36,3){$b_3$}
\put(20,3){$b_2$}
\put(4.5,3){$b_1$}
\end{picture}
\end{center}
\begin{center}
The centipede graph $W_n$
\end{center}

For this class of graphs, we will apply the results given in \cite[pp. 235]{LM} and \cite[pp. 486]{LM1} for the independence polynomial. That is, the number of independent sets of cardinality $k$ of the centipede graph is 
$$s_k=\sum\limits_{j=0}^{k}{n-j\choose n-k}{n+1-j\choose j},\ k\in\{0,1,\ldots,n\}.$$

As before, using this formula in Corollary \ref{invariants}, we obtain the following:

\begin{Proposition}
Let $G=W_n$ be the centipede graph, with $n\geq 1$, and let $I\subset T=K[s_i,t_i: 1\leq i\leq n]$ be the monomial ideal of independent sets of $W_n$, with the minimal monomial generating set $G(I)=\{m_i=\prod\limits_{r\in S_i}s_r\prod\limits_{r\notin S_i}t_r| S_i\in\Ind(W_n)\}$. Then
$$\beta_i(I)=\sum\limits_{k=0}^{n}\sum\limits_{j=0}^{k}{n-j\choose n-k}{n+1-j\choose j}{k \choose i}, \mbox{ for }i\geq 0 \mbox{ and }\projdim(T/I)=n+1.$$
\end{Proposition}

\subsection{Powers of the cycle graph}

Let $G=(V(G),E(G))$ be a graph with the vertex set $V(G)=[n]$ and let $d$ be a positive integer. We recall that the the $d-$th power of $G$ is the graph denoted by $G^d$, which have the same vertex set $V(G)$ and two distinct vertices $i$ and $j$ are adjacent in $G^d$ if and only if their distance in $G$ is at most $d$.
  
In the last years, the powers of cycles and their natural generalization, the circulant graphs, were intensively studied. We will pay attention to the $d-$th power of the cycle graph $C_n$. 

\begin{center}
%TeXCAD (http://texcad.sf.net/) Picture. File: [C^2_10.pic]. Options on following lines.
%\grade{\on}
%\emlines{\off}
%\epic{\off}
%\beziermacro{\on}
%\reduce{\on}
%\snapping{\off}
%\pvinsert{% Your \input, \def, etc. here}
%\quality{8.000}
%\graddiff{0.005}
%\snapasp{1}
%\zoom{4.0000}
\unitlength 1mm % = 2.845pt
\linethickness{0.8pt}
\ifx\plotpoint\undefined\newsavebox{\plotpoint}\fi % GNUPLOT compatibility
\begin{picture}(50,48)(0,0)
\put(28,2){\line(2,1){14}}
%\emline(42,9)(47,16.75)
\multiput(42,9)(.033557047,.052013423){149}{\line(0,1){.052013423}}
%\end
%\emline(47,29.75)(41.75,38)
\multiput(47,29.75)(-.033653846,.052884615){156}{\line(0,1){.052884615}}
%\end
\put(41.75,38){\line(-2,1){13.5}}
%\emline(28.25,44.75)(13,37.5)
\multiput(28.25,44.75)(-.070930233,-.03372093){215}{\line(-1,0){.070930233}}
%\end
%\emline(13,37.5)(8.5,29.75)
\multiput(13,37.5)(-.03358209,-.057835821){134}{\line(0,-1){.057835821}}
%\end
%\emline(8.25,16.75)(13.25,9)
\multiput(8.25,16.75)(.033557047,-.052013423){149}{\line(0,-1){.052013423}}
%\end
%\emline(13.25,9)(27.75,2.25)
\multiput(13.25,9)(.072139303,-.03358209){201}{\line(1,0){.072139303}}
%\end
\put(28,44.75){\circle*{2}}
\put(41.75,37.5){\circle*{2}}
\put(46.75,29.75){\circle*{2}}
\put(47.25,16.75){\circle*{2}}
\put(41.75,9){\circle*{2}}
\put(28,1.75){\circle*{2}}
\put(13,9){\circle*{2}}
\put(8.75,16.25){\circle*{2}}
\put(8.25,29.75){\circle*{2}}
\put(13,37.5){\circle*{2}}
\put(46.75,29){\line(0,-1){12.25}}
%\emline(46.75,16.75)(28,2.25)
\multiput(46.75,16.75)(-.0436046512,-.0337209302){430}{\line(-1,0){.0436046512}}
%\end
%\emline(28,2.25)(8.25,16.5)
\multiput(28,2.25)(-.0466903073,.0336879433){423}{\line(-1,0){.0466903073}}
%\end
%\emline(8.25,16.5)(13,37.5)
\multiput(8.25,16.5)(.033687943,.14893617){141}{\line(0,1){.14893617}}
%\end
\put(13,37.5){\line(1,0){28.5}}
\put(41.5,37.5){\line(1,-4){5.25}}
\put(8.25,29.25){\line(0,-1){13}}
%\emline(28.25,44.5)(46.75,29.75)
\multiput(28.25,44.5)(.0422374429,-.0336757991){438}{\line(1,0){.0422374429}}
%\end
%\emline(46.75,29.75)(41.25,8.75)
\multiput(46.75,29.75)(-.033536585,-.12804878){164}{\line(0,-1){.12804878}}
%\end
\put(41.25,8.75){\line(-1,0){28.25}}
%\emline(13,8.75)(8.5,29.75)
\multiput(13,8.75)(-.03358209,.156716418){134}{\line(0,1){.156716418}}
%\end
%\emline(8.5,29.75)(28.25,44.5)
\multiput(8.5,29.75)(.0450913242,.0336757991){438}{\line(1,0){.0450913242}}
%\end
\put(28,47.5){$1$}
\put(43.75,39){$2$}
\put(49,29.75){$3$}
\put(49,16.5){$4$}
\put(43.75,6){$5$}
\put(28.5,0){$6$}
\put(9.75,6){$7$}
\put(5,16.5){$8$}
\put(5,29.75){$9$}
\put(9.75,39){$10$}
\end{picture}
\\
The graph $C_{10}^2$
\end{center}

In \cite[Theorem 3.3]{BH} it was computed the independence polynomial for the powers of cycles:

\begin{Lemma}\cite{BH}
Let $n$ and $d$ be integers with $n\geq d+1$ and $d\geq 1$. Then the independence polynomial of the $d-$th power of $C_n$ is 
$$I(C_n^d;x)=\sum\limits_{k=0}^{\left\lfloor \frac{n}{d+1}\right\rfloor}\frac{n}{n-dk}{n-dk\choose dk}x^k.$$ 
\end{Lemma} 

This result allows us to compute the invariants described in Corollary \ref{invariants} for the powers of cycles:

\begin{Proposition}
Let $G=C_n^d$ be the $d-$th power of the cycle graph $C_n$, $n\geq d+1$ and $d\geq 1$. 
Let $I\subset T=K[s_i,t_i: 1\leq i\leq n]$ be the monomial ideal of independent sets. Then
$$\beta_i(I)=\sum\limits_{k=0}^{\left\lfloor \frac{n}{d+1}\right\rfloor}\frac{n}{n-dk}{n-dk\choose dk}{k \choose i}, \mbox{ for }i\geq 0 \mbox{ and }\projdim(T/I)=\left\lfloor \frac{n}{d+1}\right\rfloor+1.$$
\end{Proposition}
 
For the particular case $d=1$, we obtain the graph $G$ to be the cycle graph on $n$ vertices. Therefore, the invariants of the monomial ideal of independent sets associated to the cycle graph $C_n$ are of the form:

\begin{Corollary}
Let $G=C_n$ be the cycle graph, with $n\geq 2$ and let $I\subset T=K[s_i,t_i: 1\leq i\leq n]$ be the monomial ideal of independent sets of $G$. Then
$$\beta_i(I)=\sum\limits_{k=0}^{\left\lfloor \frac{n}{2}\right\rfloor}\frac{n}{n-k}{n-k\choose k}{k \choose i}, \mbox{ for }i\geq 0 \mbox{ and }\projdim(T/I)=\left\lfloor \frac{n}{2}\right\rfloor+1.$$
\end{Corollary}
 
	\[
\]

\end{document}